\documentclass[letter,11pt]{amsart}

\usepackage{mathrsfs,amsthm,amsmath,amssymb}
\usepackage{graphicx}

\def\PP{\mathscr{S}}
\def\EE{\mathscr{E}}
\def\FF{\mathcal{F}}
\def\OO{\mathscr{O}}
\def\CC{\mathcal{C}}
\def\Co{\mathfrak{C}}
\def\cS{\mathcal{S}}
\def\TT{\mathcal{T}}
\def\TM{\mathscr{P}}
\def\LL{{L}}
\def\AA{\mathcal{A}}

\def\zero{\widehat{0}}
\def\uno{\widehat{1}}
\def\supp{\operatorname{supp}}
\def\cl{\operatorname{cl}}

\def\ldot{\lessdot}
\def\covers{\gtrdot}
\def\op{\textup{op}}
\renewcommand\ell{l}

\newtheorem{thm}{Theorem}
\newtheorem{cor}{Corollary}
\newtheorem{lem}{Lemma}
\theoremstyle{definition}
\newtheorem{df}{Definition}
\newtheorem{qu}{Question}
\theoremstyle{remark}

\title[Modular Elimination in Matroids and Oriented Matroids]{Modular Elimination in Matroids\\ and Oriented Matroids}
\author{Emanuele Delucchi}
\address{Department of Mathematical Sciences, SUNY Binghamton, Binghamton NY 13902-6000, USA}
\email{delucchi@math.binghamton.edu}
\date{\today}


\begin{document}

\maketitle

\begin{abstract}
We introduce a new axiomatization of matroid theory that requires the elimination property only among modular pairs of circuits, and we present a cryptomorphic phrasing thereof in terms of Crapo's axioms for flats.
 
This new point of view leads to a corresponding strengthening of the circuit axioms for oriented matroids.
\end{abstract}

\section{Introduction}
In this paper we take a close look at the definition of a matroid in terms of its set of circuits. We prove that the most usual form of these axioms, requiring the {\em elimination property} to hold for each pair of circuits, is redundant. Our result shows that it is enough to require this property to hold for some special pairs of circuits: {\em modular pairs}. The main idea is to use the fact that modularity is defined in any lattice, not necessarily geometric. To make things precise, let us begin by defining the elimination property and stating the circuit axioms for matroids. For an introduction to matroid theory we point to Oxley's book \cite{Oxley}.

Given a finite set $E$, a collection  $\CC$ of subsets of $E$ and $C_1,C_2\in\CC$, we define the {\em elimination property} between $C_1$ and $C_2$ as   
\[\EE (C_1,C_2,\CC):
\textrm{ for all } e\in C_1\cap C_2 \textrm{ there is } C_3\in\CC\textrm{ with }C_3\subseteq (C_1\cup C_2)\setminus \{e\}.
\]

\begin{df}\label{def:M}
Let $E$ be a finite set. A collection $\CC$ of incomparable nonempty subsets of $E$ is the {\em set of circuits of a matroid} on the ground set $E$ if $\EE(C_1,C_2,\CC)$ holds for all $C_1,C_2\in\CC$.
\end{df}

One important feature of matroid theory is the availability of many cryptomorphic axiomatizations, different in spirit but equivalent in substance. For example, notice that the set of all unions of circuits, partially ordered by inclusion, is an inverted geometric lattice (see \cite[Chapter 1.7]{Oxley}). In fact, the structure of this lattice does encode the full matroid structure.

Given the set $\CC$ of circuits of a matroid, let us call a pair of circuits $A,B\in \CC$ a {\em modular pair} if $A\vee B$ has rank $2$ in the associated lattice (see for instance \cite{fuoriorario}). The key observation is then that for this definition to make sense all we need to know about the family $\CC$ is that its members are incomparable (compare Definition \ref{def:U}). \\

We now consider oriented matroids;
a general introductory reference is \cite{BLSWZ}. We consider not only subsets of $E$, but {\em signed subsets}, i.e., functions $X:E\to \{-1,0,+1\}$ representing the ``signature'' of the set given by the support of $X$. Notice that the set of signs has a natural $\mathbb Z_2$ action (switching sign). 
The {\em support} of a signed set $X$ is $\supp(X):=\{e\in E\mid X(e)\neq 0\}$, and we will call a collection $\Co\subseteq\{-1,0,+1\}$ {\em simple} if $\supp(Y)=\supp(X)$ implies $X=\pm Y$ for all $X,Y\in \Co$. 
In order to give a definition that exhibits the similarity with the previous one for matroids, let us define a ``oriented elimination'' property for any $\Co\subseteq\{- 1, 0, +1\}^E$ and any $X,Y\in \Co$.
\[
\begin{array}{rl}
\OO\EE (X,Y,\Co):&
\!\!\textrm{for all } e,f \textrm{ with } X(e)=-Y(e)\neq 0, X(f)\neq Y(f)\\ &\!\!\textrm{there is } Z\in\Co \textrm{ with } Z(e)=0, Z(f)\neq 0,\\
&\!\!\textrm{and } Z(g)\in\{0,X(g),Y(g)\} \textrm{ for all }g\in E.
\end{array}
\]
 
Then, one definition of oriented matroids is the following.
\begin{df}[see \cite{BLSWZ}]\label{OM:def} A $\mathbb Z_2$-symmetric, simple  collection $\Co\subseteq\{-1,0,+1\}^E$ 
is the set of signed circuits of an oriented matroid if
\begin{itemize}
\item[(1)] the collection $\CC:=\{\supp(X)\mid X\in \Co\}$ is the set of circuits of a matroid on $E$, and
\item[(2)] $\OO\EE(X,Y,\Co)$ holds for all $X,Y\in \Co$ such that $\supp(X),\supp(Y)$ are a modular pair in $\CC$.
\end{itemize}
\end{df}

The previous definition is usually presented as an interesting ``curiosum'', whereas the standard definition requires $\OO\EE(X,Y,\Co)$ to hold for every pair of elements $X,Y\in\Co$. Recent work by Laura Anderson and the author~\cite{AnDe} led to the observation that for linear dependencies in complex spaces an analogous ``oriented elimination'' with ocmplex signs (i.e.,\ phases of complex numbers) - cannot be expected to hold for all pairs of circuits but only for modular pairs (see \cite[Sections 3 and 6]{AnDe} for details). This pointed to the possibility that (unoriented) elimination among modular pairs of circuits could be a defining property for matroids.\\


Our Theorem \ref{main} shows that indeed, given a collection of incomparable nonempty subsets of a finite ground set, requiring the elimination property for modular pairs  (defined as in Definition \ref{def:U}) of elements of this collection is enough to ensure that the collection satisfies the circuit axioms for matroids. As a corollary, we can remove condition (1) from Definition \ref{OM:def}. With Theorem \ref{thm:newcrapo} we then also describe a cryptomorphic weakening of Crapo's axioms for flats.

\subsection*{Acknowledgments} The author would like to thank Thomas Zaslavsky and Fernando Guzm\`an for helpful discussions during the preparation of the manuscript. The idea for the paper arose during the joint work with Laura Anderson on \cite{AnDe}.

\section{Main result}

\noindent {\em Notation and basics.} As a general reference on the combinatorics of posets and lattices we refer to \cite[Chapter 3]{Stanley}. Here let us only recall that a {\em chain} $J$ in a poset $(P,\leq)$ is any totally ordered subset of $P$; the {\em length} of the chain $J$ is then $\ell(J):=\vert J \vert -1$. Given $x\in P$ we write $P_{\geq x} =\{x'\in P\mid x'\geq x\}$ and $P_{\leq x}=\{x'\in P\mid x'\leq x\}$. The {\em length} of $P$ is $\ell(P):=\max\{\ell(J) \mid J \textrm{ a chain of }P\}$, and for $x\in P$ write $\ell(x):=\ell(P_{\leq x})$.  

Given two elements $x,y\in P$, we say that {\em $x$ covers $y$}, written $x \ldot y$, if $x\leq y$ and $\vert P_{\geq x} \cap P_{\leq y}\vert =2$.

 If for any $x,y\in P$ the poset $P_{\geq x}\cap P_{\geq y}$ has a unique minimal element, this element is denoted $x\vee y$ and called the {\em meet} of $x$ and $y$. Analogously we call $x\wedge y$, or {\em join} of $x$ and $y$, the unique maximal element of $P_{\leq x}\cap P_{\leq y}$, if it exists.  The poset $P$ is called a {\em lattice} if meet and join are defined for every pair of elements of $P$. In particular, every finite lattice has a unique minimal element, called $\zero$, and a unique maximal element, called $\uno$. In any poset with a unique minimal element $\zero$, the elements $a$ with $\zero\ldot a$ are called {\em atoms}.

\begin{df}\label{def:U} Let $L$ be a lattice. The {\em atoms} of $L$ are the elements that cover $\zero$ in $L$. The lattice $L$ is called {\em atomic} if every $x\in L$ is $x=\bigvee A$ for some set $A$ of atoms of $L$. We say that two atoms $a,b$ of $L$ form a {\em modular pair} if $\ell(L_{\leq a\vee b})=2$. 

Given any family $\cS$ of subsets of a set $E$, consider the set 
\[U(\cS):=\{\bigcup \TT \mid \TT\subseteq \cS\}\]
partially ordered by inclusion - so, for $A,B\in U(\cS)$, $A\leq B$ if $A\subseteq B$. 

If the members of $\cS$ are incomparable, then $U(\cS)$ is a lattice with $\zero=\emptyset$ where join and meet of any two elements $A,B\in U(\cS)$ are defined by $A\vee B:= A\cup B$ and $A\wedge B:=\bigcup\{S\in \cS \mid S\subseteq A\cap B\}$.  This lattice is atomic by definition. We will say that two members of $\cS$ are a {\em modular pair} if they are a modular pair in $U(\cS)$.
\end{df}

\begin{thm}\label{main}
Let $\CC$ be a collection of incomparable finite subsets of a set $E$.
If $\EE(A,B,\CC)$ for all modular pairs $A,B\in \CC$, then $\EE(A,B,\CC)$ for all pairs $A,B\in \CC$.
\end{thm}
\begin{proof} 
Take $A, B\in \CC$ with $A\neq B$,  $e\in A\cup B$ and let $Z:= A\cup B = A\vee B$. We want to show that a $C\in \CC$ exists with $ e\not\in C\subseteq Z $. The finiteness requirement on the cardinality of the elements of $\CC$ ensures that $\ell(X)<\infty$ for all $X\in U(\CC)$. We will proceed by induction on $\ell(Z)$.

 If $\ell(Z)=2$ then $A,B$ is a modular pair and we are done. Suppose now $\ell(Z)=n>2$ and let $J$ be a chain of maximal cardinality in $U(\CC)_{\leq Z}$. The chain $J$ contains exactly one element $A'\in\CC$ and at least an element $Y$ with $A'\lneq Y \lneq Z$. If $e\not\in A'$ we are done with $C:=A'$. Else, since $U(\CC)$ is atomic, there is $B'\in\CC$ with $A'\vee B' \leq Y$. Again, if $e\not\in B'$ then $C:=B'$ does it; otherwise $e\in A'\cap B'$ and we may apply the inductive hypothesis to the pair $A',B'$ (because $Y<Z$ implies $\ell(Y)<\ell(Z)$), obtaining $C$ as desired.
\end{proof}


Restricting $E$ to be a finite set, Theorem \ref{main} gives the desired result.
\begin{cor}\label{M:thm}
A collection $\CC$ of incomparable subsets of a finite set $E$ is the set of circuits of a matroid on $E$ if $\EE(A,B,\CC)$ for all modular pairs $A,B\in\CC$.
\end{cor}

As a straightforward consequence we have a corresponding strengthening
of the axiomatics of oriented matroids given in Definition \ref{OM:def}.

\begin{cor}\label{OM:thm} A $\mathbb
  Z_2$-symmetric, simple
  collection $\Co$ of elements of $\{-1,0,+1\}^E$ with incomparable support   
is the set of signed circuits of an oriented matroid if and only if 
$\OO\EE(X,Y,\Co)$ holds for all $X,Y\in \Co$ such that
$\supp(X),\supp(Y)$ are a modular pair in the set of supports of
elements of $\Co$.
\end{cor}
\begin{proof} From Corollary \ref{M:thm} we know that, under the hypotheses of the theorem, $\CC:=\{\supp(X)\mid X\in\Co\}$ is the set of circuits of a matroid (indeed, $\OO\EE(X,Y,\Co)$ implies $\EE(\supp{X},\supp{Y},\CC)$). 
%
\end{proof}

In its full generality, Theorem \ref{main} implies the corresponding strengthening of the axioms for  {\em finitary matroids} (also called {\em independence spaces} - see \cite[Chapter 20]{Welsh} for an overview, \cite{InfOxley} for an extended account).

\begin{cor} A family $\CC$ of incomparable finite subsets of a (possibly infinite) set $E$ is the family of circuits of a finitary matroid if and only if $\EE(A,B,\CC)$ holds for all modular pairs $A,B\in\CC$.
\end{cor}

\section{Flats and geometric lattices}\label{sec:geolat}

A matroid given by its set of circuits $\CC$ as in Definition \ref{def:M} gives rise to a {\em closure operator} on its ground set $E$ 
  $$
    \cl : \TM(E)\to \TM(E),\quad A\mapsto \cl(A):=A\cup\{e\in E \mid e\in C\subseteq A\cup\{e\}\textrm{ for a }C\in \CC\}
  $$ 

\noindent where $\TM(E)$ is the power set of $E$ (compare \cite[Proposition 1.4.10]{Oxley}).

A set $X\subseteq E$ is {\em closed} if $\cl(X)=X$. Closed sets of matroids are usually called {\em flats}. The collection of flats of a matroid, partially ordered by inclusion, is a lattice.  After a preparatory definition we will state a characterization, due to Crapo \cite{Crapo}, of the posets that arise as lattices of flats of a matroid.

\begin{df}
Consider a finite poset $P$ and let $\AA$ be its set of atoms. Given $x\in P$, write 
$\AA_x:=\AA\cap P_{\leq x}$. We say that $x$ satisfies Crapo's {\em separation property}  (essentially axiom $\beta$ in \cite{Crapo}) if the property
 $$ 
   \PP(x,P):= \textrm{ }  \{\AA_{x'}\setminus \AA_{x}\}_{x'\covers x} \textrm{ is a partition of } \AA\setminus \AA_{x}
 $$
is satisfied.
\end{df}

\begin{lem}[Crapo's axioms for flats, see \cite{Crapo} or p.\ 35 of \cite{Oxley}]
\label{thm:crapo}
Let $E$ be a finite set. An intersection-closed family $\FF$ of subsets of $E$ with $E \in\FF$, partially ordered by inclusion, is the set of flats of a matroid on the ground set $E$ if and only if $\PP(X,\FF)$ holds for all $X\in \FF$.
\end{lem}

We now come to our strengthening of Crapo's axioms. Under a mild additional assumption on the family $\FF$ we show that it is enough to check property $\PP$ on coatoms only. Call a collection $\FF\subseteq \TM(E)$ {\em intersection-generated} if all its elements are intersections of the maximal elements of $\FF \setminus E$.

\begin{thm}\label{thm:newcrapo}
Let $E$ be a finite set. An intersection-generated family $\FF$ of subsets of $E$ with $E\in\FF$, ordered by inclusion, is the set of flats of a matroid on the ground set $E$ if and only if $\PP(X,\FF)$ holds for all $X\in\FF$ with $\ell(\FF_{\geq X}) = 2$. 
\end{thm}
\begin{proof} If $\FF$ is the lattice of flats of a matroid then, by Lemma
  \ref{thm:crapo},  $\PP(X,\FF)$ holds for all $X\in \FF$. 

To prove the other direction, let $\FF$ be an intersection-generated collection of subsets of $E$ such that $E\in \FF$. Then, with respect to the partial order given by inclusion, $\FF^{\op}$ is an atomic lattice. Its set of atoms is 
$$
 \CC:=\{E\setminus F \mid F\subseteq \FF,\, F\ldot E \}
$$
which is a collection of incomparable subsets of $E$ such that $U(\CC)=\FF^{{\op}}$.
We collect the following two facts, which are easily checked by complementation.
\begin{itemize}
\item[(1)] $A,B\in\CC$ are a modular pair if and only if $E\setminus A
  \covers F$ and $E\setminus B\covers F$ for some $F$ with
  $\ell(\FF_{\geq F})=2$. 
\item[(2)] For a modular pair $A,B$ in $\CC$, 
  $\PP((E\setminus A) \wedge (E\setminus B) , \FF)$ implies $\EE(A,B,\CC)$.  
\end{itemize}

By (2), $\EE(A,B,\CC)$ holds for a given modular pair $A,B$ of
$\CC$ if $\PP(X,\FF)$ holds for some $X$ that, by (1), has
$\ell(\FF_{\geq X})=2$. Thus, if $\PP(X,\FF)$ holds for all $X\in\FF$
with $\ell(\FF_{\geq X})=2$, $\CC$ satisfies the hypotheses of
Corollary \ref{M:thm} and so it is the set of circuits of a
matroid. The members of $\FF$ are complements of members of $U(\CC)$
and so, e.g. by \cite[p.\ 78]{Oxley}, they are flats of another
matroid (dual to the former). In particular, $\FF$ is the lattice of flats of some matroid.
\end{proof}

\begin{df} A finite atomic lattice $\LL$ is {\em geometric} if it is the lattice of flats of a matroid.
\end{df}
As the title of the paper \cite{Crapo} itself says, Crapo was interested in what we called property $\PP$ as a means for characterizing the structure of {geometric lattices}. We thus close in the same spirit by stating a corollary of Theorem \ref{thm:newcrapo}, which strengthens \cite[Proposition 4]{Crapo} in the finite case.

\begin{cor}\label{cor_lattice} A finite lattice $\LL$ is {geometric} if and only if $\PP(x,\LL)$ holds for all $x\in \LL$ with $\ell(\LL_{\geq x})=2$.
\end{cor}
\begin{proof}
Every finite atomic lattice $\LL$ is order-isomorphic to the poset obtained by ordering the family of sets $\FF:=\{\AA_x\}_{x\in\LL}$ by inclusion. Moreover, $\PP(\AA_x,\FF)$ if and only if $\PP(x,\LL)$ for all $x\in \LL$. Now Theorem \ref{thm:newcrapo} shows that $\FF$ is the lattice of flats of a matroid, hence geometric, and thus so is $\LL$ as well. 
\end{proof}

\begin{qu}
In the previous section we gave a corollary of Theorem \ref{main} that
stated our result for infinite matroids in terms of one possible
approach via ``circuits''. Since duality is precisely one of the
features that independence spaces lack \cite[Theorem
3.1.13]{InfOxley}, it is not possible to use the argument of Theorem
\ref{thm:newcrapo} in the infinite case. 
However, since infinite geometric lattices are well-studied objects
that deserve interest in their own right (see for example \cite{MaMa}), we ask the following question: {\em Is there a
  lattice-theoretic (``dual'') version of Theorem \ref{main} that
  would lead to an analogue of Corollary \ref{cor_lattice} for
  infinite lattices?}
\end{qu}

\bibliographystyle{abbrv}

\bibliography{MEbib}

\end{document}